\documentclass[english]{ccdconf}
\usepackage[comma,numbers,square,sort&compress]{natbib}
\usepackage{epstopdf}
\usepackage{amsmath}
\usepackage{amsfonts}
\usepackage{amssymb}
\usepackage{graphicx}
\usepackage{amsmath,amssymb,amsfonts}
\usepackage{algorithmic}
\usepackage{graphicx}
\usepackage{textcomp}
\usepackage{xcolor,mathdots}
\usepackage{hhline}
\usepackage{multirow}
\usepackage{makecell}
\usepackage{amsthm}
\newtheorem{theorem}{Theorem}

\newtheorem{proposition}{Proposition}

\begin{document}

\title{A modified Riccati approach to analytic interpolation with applications to system identification and robust control}
\author{Yufang Cui\aref{sjtu},
	    Anders Lindquist\aref{sjtu}}

\affiliation[sjtu]{Shanghai Jiao Tong University, Shanghai, China
        \email{cui-yufang@sjtu.edu.cn},\, {alq@kth.se}}

\maketitle

\begin{abstract}
This paper provides a new method to solve analytic interpolation problems with rationality and derivative constraints, occurring in many applications to system and control. It is based on the covariance extension equation previously proposed by Byrnes and Lindquist in a different context. A complete solution for the scalar problem is provided, and a homotopy continuation method is presented and applied to some problems in modeling and robust control. Some numerical examples illustrate robustness and efficiency of the proposed procedure.
\end{abstract}


\section{Introduction}\label{Introduction}
Many important questions in systems and control can be formulated as an analytic interpolation problem, which in its most general (scalar) form can be formulated in the following way.  Given $m+1$ distinct complex numbers $z_0,z_1,\dots,z_m$ in the complement $\mathbb{D}^{C}:=\{z \mid |z|>1\}$ of the closed unit disc (possibly including $z=\infty$), find a strictly {\em positive real\/} function $f$, i.e., a function $f$ that is analytic in $\mathbb{D}^{C}$ and satisfies $\text{Re}\{f(z)\}>0$ there, which satisfies the interpolation conditions 
\begin{align}
\label{interpolation}
 \frac{f^{(k)}(z_{j})}{k!}=v_{jk},\quad&j=0,1,\cdots,m,   \\
    &   k=0,\cdots n_{j}-1 \notag
\end{align}
(with $f^{(k)}$ the $k$:th derivative),  and which is rational of degree at most 
\begin{equation}
\label{deg(f)}
n:=\sum_{j=0}^{m}n_j -1.
\end{equation}
To simplify calculations, we shall normalized the problem by setting $z_0=\infty$ and $ f(\infty)=\frac{1}{2} $, which can be achieved through a simple M{\"o}bius transformation. Moreover, we assume that $f$ is a real function. This implies that $f^{(k)}(\bar{z}_j)/ k!=\bar{v}_{jk}$ is an interpolation condition whenever $f^{(k}(z_j)/ k!=v_{jk}$ is.

With $m=0$ and $n_0=n+1$,  this reduces to the {\em rational covariance extension problem\/} introduces by Kalman \cite{Kalman-81} and completely solved in steps in \cite{Gthesis,G87,BLGuM,BLpartial,BGuL}. In fact, this problem, which amounts to finding a rational positive real functions of prescribed maximal degree given a partial covariance sequence, is a basic problem in signal processing and speech processing \cite{b12} and system identification \cite{b13,LPbook}.

If $n_0=n_1=\dots=n_m=1$, i.e., the interpolation points are simple and distinct, we have the regular {\em Nevanlinna-Pick interpolation problem with degree constraint\/} \cite{b15,b1,BLkimura} occurring in robust control, high-resolution spectral estimation, simultaneous stabilization and many other problems in systems and control. In fact, the Nevanlinna-Pick interpolation problem to find a positive real function that interpolates the given data was early used in systems and control \cite{b9,b10} and show obvious advantages in spectral estimation \cite{b2}. 

The general Nevanlinna-Pick interpolation problem with degree constraint allowing derivative constraints, described above, was studied in detail in \cite{blomqvist}. This study was motivated by $H^\infty$ control problems with multiple unstable poles and/or zeros in the plant, problems that could not be handled by a classical interpolation approach \cite[p. 18]{GreenLimebeer}. 

The proof in the early work on the rational covariance extension problem \cite{Gthesis,G87} and the complete smooth parameterization of all solutions \cite{BLGuM} were nonconstructive, using topological degree theory. A first attempt  to provide an algorithm was presented by Byrnes and Lindquist \cite{BLpartial}, where a new nonstandard Riccati-type equation called the Covariance Extension Equation (CEE) was introduced. This approach was completely superseded by a convex optimization approach \cite{BGuL,b1}, and thus abandoned. However, in a brief paper \cite{b6}, Lindquist indicated that the regular Nevanlinna-Pick interpolation problem with degree constraint could also be solved by the Covariance Extension Equation, and thus he showed that CEE  is universal in the sense that it can be used to solve more general analytic interpolation problems by only changing certain parameters. 

In this paper we  take such an approach to the general Nevanlinna-Pick interpolation problem with both degree and derivative constraints, and we shall provide a homotopy continuation method to solve the corresponding CEE. It turns out this procedure is quite efficient and numerically robust. It also has the advantage of easily detecting when model reduction is possible without reducing accuracy. 

The paper is organized as follows. In Section~\ref{sec:CCE} we review useful facts about the Covariance Extension Equation and the context in which it was first presented.  Section~\ref{sec:general} presents a derivation of the CEE in the context of our new general interpolation problem.  Section~\ref{sec: theorems} presents the fundamental main theorems on existence and uniqueness of solutions and  the basic diffeomorphism results needed for homotopy continuation, 
used in Section~\ref{sec:homotopy} to develop our computational procedure.  
In Section~\ref{sec:applications}, finally, we apply our method to some problems in identification and robust control.

\section{The Covariance Extension Equation}\label{sec:CCE}

Since $f$ is analytic in $\mathbb{D}^{C}$ and $f(\infty)=\tfrac12$, there is an expansion
\begin{equation}\label{f}
f(z)=\frac{1}{2}+c_{1}z^{-1}+c_{2}z^{-2}+c_{3}z^{-3}+\cdots ,
\end{equation}
and, since $f$ is positive real, 
\begin{equation}
\Phi(z):=f(z)+f(z^{-1})=\sum_{k=-\infty}^{\infty}c_{k}z^{-k}>0 \quad z\in \mathbb{T},
\end{equation}
where $\mathbb{T}$ is the unit circle $\{ z=e^{i\theta}\mid 0\leq\theta <2\pi\}$. Hence $\Phi$ is a power spectral density, and therefore there is a minimum-phase spectral factor $w(z)$ such that 
\begin{equation}
\label{spectralfactor}
w(z)w(z^{-1})=\Phi(z).
\end{equation}
Clearly $f$ has a representation
\begin{equation}\label{ab2f}
f(z)=\frac{1}{2}\frac{b(z)}{a(z)} 
\end{equation}
where
\begin{subequations}
\begin{equation}
a(z)=z^{n}+a_{1}z^{n-1}+\cdots+a_{n}\\
\end{equation}
\begin{equation}
b(z)=z^{n}+b_{1}z^{n-1}+\cdots+b_{n}\\
\end{equation}
\end{subequations}
are Schur polynomials, i.e., polynomials with all roots in the open unit disc $\mathbb{D}$. Consequently 
\begin{equation}
\label{GG*}
w(z)w(z^{-1})=\frac{1}{2}\left[\frac{b(z)}{a(z)} +\frac{b(z^{-1})}{a(z^{-1})}\right],
\end{equation}
and therefore 
\begin{equation}
\label{G}
w(z)=\rho\frac{\sigma(z)}{a(z)},
\end{equation}
where $\rho>0$ and 
\begin{equation}\label{sigma}
\sigma(z)=z^{n}+\sigma_{1}z^{n-1}+\cdots+\sigma_{n} 
\end{equation}
is a Schur polynomial. It follows from \eqref{GG*} and \eqref{G} that
\begin{equation}
\label{ab2sigma}
a(z)b(z^{-1})+b(z)a(z^{-1})=2\rho^{2}\sigma(z)\sigma(z^{-1}).
\end{equation}
We shall represent the monic polynomials $a(z)$, $b(z)$ and $\sigma(z)$ by the $n$-vectors 
\begin{equation}
\label{absigma}
\begin{bmatrix}a_1\\a_2\\\vdots\\a_n\end{bmatrix},
\quad
\begin{bmatrix}b_1\\b_2\\\vdots\\b_n\end{bmatrix}\quad\text{and}
\quad
\begin{bmatrix}\sigma_1\\\sigma_2\\\vdots\\\sigma_n\end{bmatrix}.
\end{equation}

Following \cite{b6} we note that \eqref{ab2f} has an observable realization
\begin{equation}
f(z)=\frac{1}{2}+h'(zI-F)^{-1}g
\end{equation}
where  
\begin{subequations}
\begin{equation}\label{g}
F=J-ah',\quad  g=\frac{1}{2}(b-a),
\end{equation}
\begin{equation}\label{J}
h=\begin{bmatrix}1\\0\\\vdots\\0\end{bmatrix},
\quad
J=\begin{bmatrix}
0&1&0&\cdots&0\\
0&0&1&\cdots&0\\
\vdots&\vdots&\vdots&\ddots&\vdots\\
0&0&0&\cdots&1\\
0&0&0&\cdots&0\\
\end{bmatrix}.
\end{equation}
\end{subequations}

From stochastic realization theory  \cite[Chapter 6]{LPbook} it follows that the minimum-phase spectral factor \eqref{G} has a realization
\begin{equation}
w(z)=\rho+h'(zI-F)^{-1}k
\end{equation}
where
\begin{equation}
\rho=\sqrt{1-h'Ph},\quad k=\rho^{-1}(g-FPh)
\end{equation}
with $P$ being the minimum solution of the algebraic Riccati equation
\begin{equation}\label{ARE}
P=FPF'+(g-FPh)(1-h'Ph)^{-1}(g-FPh)' .
\end{equation}
Following the calculations in \cite{BLpartial,b6} we now see that
\begin{equation}\label{gk}
g=\Gamma Ph+\sigma-a,\quad k=\rho(\sigma-a) 
\end{equation}
and that \eqref{ARE} can be reformulated as 
\begin{equation}\label{P}
P=\Gamma(P-Phh'P)\Gamma'+gg' 
\end{equation}
where $\Gamma$ is given by 
\begin{equation}
\label{Gamma}
\Gamma =J-\sigma h' .
\end{equation}

The  rational covariance extension problem, i.e., the special case $m=0$ and $v_{0k}=c_k$ for $k=0,1,\dots, n$ in the general problem \eqref{interpolation}, amounts to finding $(a,b)$ given a partial covariance sequence $c:=(c_0,c_1,\dots,c_n)$ and a particular Schur polynomial $\sigma(z)$.  In \cite{BLpartial} it was shown that the {\em Covariance Extension Equation (CEE)}
\begin{subequations}\label{PgCCE}
\begin{equation} \label{CEE}
 P = \Gamma (P-Phh'P) \Gamma' + g(P)g(P)' 
\end{equation}
(where $^\prime$ denotes transposition) with
\begin{equation}\label{g(P)} 
g(P)= u +U\sigma + U\Gamma Ph ,
\end{equation}
\end{subequations}
where $u:=(u_1,u_2,\dots,u_n)'$  and the matrix $U$ are determined from the expansion
\begin{equation*}
\frac{z^n}{z^n + c_1 z^{n-1}+\dots + c_n}= 1 - u_1 z^{-1} - u_2z^{-2}  - \dots ,
\end{equation*}
has a unique symmeric solution $P\geq 0$ such that $h'Ph<1$. 
Moreover, for each $\sigma$ there is a unique solution of the rational covariance extension problem, and it is given by
\begin{subequations}\label{Psigma2ab}
\begin{equation}\label{a}
a=(I-U)(\Gamma Ph+\sigma)-u
\end{equation}
\begin{equation}\label{b}
b =(I+U)(\Gamma Ph+\sigma)+u
\end{equation}
\begin{equation}\label{rho}
\rho=\sqrt{1-h'Ph} ,
\end{equation}
\end{subequations}
and 
the degree of $f(z)$ equals the rank of $P$.
This rank condition is very useful in modeling, since small singular values of the solution $P$ indicates that reasonable model reduction is possible. 

One of the main results of this paper is to show that  CEE can also be used to solve the general Nevanlinna-Pick interpolation problem with degree and derivative constraints presented above by merely changing the matrix $(u,U)$.

\section{The general interpolation problem}\label{sec:general}

To simplify the problem we reformulate the problem by considering instead of $f$  
\begin{equation}
\label{varphi}
\varphi(z):=f(z^{-1})= \frac12\frac{b_{*}(z)}{a_{*}(z)},
\end{equation}
where $a_{*}(z):=z^n a(z^{-1})$ is the reversed polynomial. Since $f$ is positive real, $\varphi$ is a Carath\'eodory function mapping the unit disc $\mathbb{D}$ to the right half-plane. The new interpolation points $(z_0,z_1,\dots,z_m)$ are now obtained via the transformation $z_j^{-1}\to z_j$, and in particular $z_0=0$. Then the interpolation conditions \eqref{interpolation} becomes 
\begin{align}
\label{interpolation2}
 \frac{\varphi^{(k)}(z_{j})}{k!}=w_{jk},\quad&j=0,1,\cdots,m,   \\
    &   k=0,\cdots n_{j}-1 \notag
\end{align}
where the values $w_{jk}$ are given by
\begin{subequations}\label{w_jk}
\begin{align}
  w_{0k}  &= v_{0k}, \quad k=0,1,\cdots,n_0-1 \\
  w_{j0}  &= v_{j0}, \quad   j=1,2,\cdots,m\\
w_{jk}  &= \sum_{\ell=1}^{k}\frac{\ell!(k-\ell+1)!}{k!}s_{k}^\ell v_{j,k-\ell+1}(-1)^{k+2}(z_{j})^{2k-\ell+1}\\&\quad j=1,\cdots,m,\, k=1,\cdots,n_{j}-1 ,\notag
\end{align}
\end{subequations}
and where 
\begin{equation}
\begin{split}
&s_{k}^{1}=s_{k}^{k}=1,\quad k=1,\cdots,n_{j}-1,\\
&s_{k+1}^{\ell}=\frac{2k-\ell+2}{\ell}s_{k}^{\ell -1}+s_{k}^\ell,\quad \ell=2,\cdots,k .
\end{split}
\end{equation}
(Cf. \cite{blomqvist}.) Then, given \eqref{w_jk}, we form
\begin{equation}
\label{W}
W:=\begin{bmatrix}
W_{0}&~&~\\
~&\ddots&~\\
~&~&W_{m}
\end{bmatrix},
\end{equation}
where, for $j=0,1,\dots,m $,
\begin{equation}\label{Wj}
W_{j}=\begin{bmatrix}
w_{j0}&~&~&~\\
w_{j1}&w_{j0}&~&~\\
\vdots&\ddots&\ddots&~\\
w_{jn_{j-1}}&\cdots&w_{j1}&w_{j0}
\end{bmatrix}
\end{equation}

Next define the $n+1$-dimensional column vector
\begin{equation}\label{e}
e:=[e_{1}^{n_{0}},e_{1}^{n_{1}},\cdots,e_{1}^{n_{m}}]',
\end{equation}
where $e_{1}^{n_{j}}=[1,0,\cdots,0]\in\mathbb{R}^{n_j}$ for each $j=0,1,\dots,m$.
Moreover, set 
\begin{equation}\label{Z}
Z:=\begin{bmatrix}
Z_{0}&~&~\\
~&\ddots&~\\
~&~&Z_{m}
\end{bmatrix},\; Z_{j}=\begin{bmatrix}
z_{j}&~&~&~\\
1&z_{j}&~&~\\
~&\ddots&\ddots&~\\
~&~&1&z_{j}
\end{bmatrix}
\end{equation}
Since $\lambda(Z)<1$, the Lyapunov equation
\begin{equation}
E=ZEZ^{*}+ee^{*}
\end{equation}
has a unique solution $E$. Here $Z^*$ is the the Hermitian conjugate (transposition + conjugation). We refer  to  \cite{b16,blomqvist,b1} for the following result.

\begin{proposition}
There exists a (strict) Carath\'eodory function $\varphi$ satisfying \eqref{interpolation2}, or equivalently a strictly positive real function $f$ satisying \eqref{interpolation}, if and only if 
\begin{equation}
\label{Pick}
\Sigma =WE+EW^*
\end{equation}
is positive definite.
\end{proposition}
The matrix $\Sigma$ is called the {\em generalized Pick matrix}.

Since 
\begin{equation}
\label{varphi(Z)}
\varphi(Z)=\frac{1}{2}I+c_{1}Z+c_{2}Z^{2}+c_{3}Z^{3}\cdots=W
\end{equation}
\cite{Higham} and $b_{*}(Z)=2\varphi(Z)a_{*}(Z)$,
\begin{equation}
\label{ }
b_{*}(Z)e=2Wa_{*}(Z)e
\end{equation}
and consequently
$$V\begin{bmatrix}1\\b\end{bmatrix}=2WV\begin{bmatrix}1\\a\end{bmatrix}$$
where
$$V:=[e,Ze,Z^{2}e,\cdots,Z^{n}e].$$
Therefore, since $V$ is invertible, 
\begin{equation}\label{a2b}
\begin{bmatrix}
1\\b\end{bmatrix}=2V^{-1}WV\begin{bmatrix}
1\\a\end{bmatrix}
\end{equation}
From \eqref{g} we have $g=\tfrac{1}{2}(b-a)$, which implies that 
\begin{equation}\label{a2g}
\begin{bmatrix}0\\g\end{bmatrix}=T\begin{bmatrix}1\\a\end{bmatrix},
\end{equation}
where 
\begin{equation}
T=\frac{1}{2}(2V^{-1}WV-I) ,
\end{equation}
or equivalently
\begin{equation}
\label{a+g}
(I+T)\begin{bmatrix}0\\g\end{bmatrix}=T\begin{bmatrix}1\\a+g\end{bmatrix}.
\end{equation}
Now,
\begin{displaymath}
I+T=V^{-1}WV+\frac{1}{2}I=V^{-1}(W+\frac{1}{2}I)V
\end{displaymath}
is nonsingular, and therefore \eqref{a+g} and \eqref{gk} yield
\begin{equation}
\label{g2Psigma}
\begin{bmatrix}0\\g\end{bmatrix}=(I+T)^{-1}T\begin{bmatrix}1\\\Gamma Ph+\sigma\end{bmatrix}
\end{equation}
Define
\begin{align}\label{uU}
\begin{bmatrix}
u&U\end{bmatrix}:&=\begin{bmatrix}0&I_{n}\end{bmatrix}(I+T)^{-1}T\notag\\
&=\begin{bmatrix}0&I_{n}\end{bmatrix}V^{-1}DV,
\end{align}
where
\begin{equation}
\label{D}
D:=(W+\frac{1}{2}I)^{-1}(W-\frac{1}{2}I)
\end{equation}
and where $I_n$ denotes the $n\times n$ identity matrix to distinguish it from the $(n+1)\times (n+1)$ identity matrix $I$.  Then \eqref{g2Psigma} yields
\begin{equation} \label{g1}
g=u+U\sigma+U\Gamma Ph  
\end{equation}
where $u$ is an $n$ vector and $U$ an $n\times n$ matrix. Inserting \eqref{g1} into \eqref{P}, we have
\begin{equation}\label{newCCE}
\begin{split}
P=&\Gamma(P-Phh'P)\Gamma'\\
&+(u+U\sigma+U\Gamma Ph)(u+U\sigma+U\Gamma Ph)' ,
\end{split}
\end{equation}
which is precisely the Covariance Extension Equation (CEE) \eqref{PgCCE}, but now with $(u,U)$ exchanged for \eqref{uU}. Moreover, by (\ref{gk}) and (\ref{g}),
\begin{equation}\label{P2ab}
\begin{split}
a&=(I-U)(\Gamma Ph+\sigma)-u\\
b&=(I+U)(\Gamma Ph+\sigma)+u
\end{split}
\end{equation}
in harmony with \eqref{Psigma2ab}.
Let  the first column in \eqref{Wj} be denoted $(w_{j0}, w_j')'$ and form the $n$-vector 
\begin{displaymath}
w=(w_0',w_{10}, w_1',w_{20}, w_2',\dots,w_{m0}, w_m')',
\end{displaymath}
where $w_{00}=\tfrac12$ has been removed since it is a constant and not a variable.

\begin{proposition}\label{w2uUprop}
There is map $u=\omega(w)$ sending $w$ to $u$, which is a diffeomorphism. Moreover, there is a linear  map $L$ such that $U=Lu$.
\end{proposition} 

\begin{proof}
Partition the matrix $W_j +\tfrac12 I$ as 
\begin{displaymath}
W_j +\frac12 I =\begin{bmatrix}w_{j0}+\frac{1}{2}&0\\w_j&C_j\end{bmatrix}
\end{displaymath}
and use the inversion formula 
$$\begin{bmatrix}
A&0\\
C&D
\end{bmatrix}^{-1}=\begin{bmatrix}
A^{-1}&0\\
-D^{-1}CA^{-1}&D^{-1}
\end{bmatrix}$$
to obtain
\begin{align*}
D_j:=&(W_{j}+\frac{1}{2}I)^{-1}(W_{j}-\frac{1}{2}I)\\
=&\begin{bmatrix}
 (w_{j0}+\frac{1}{2})^{-1}(w_{j0}-\frac{1}{2})&0\\
 C_j^{-1}w_j(w_{j0}+\frac{1}{2})^{-1}&C_j^{-1}(C_j-I)
  \end{bmatrix},
 \end{align*}
  Setting
  \begin{displaymath}
 \begin{bmatrix}d_{j0}\\d_j\end{bmatrix}= \begin{bmatrix}(w_{j0}+\frac{1}{2})^{-1}(w_{j0}-\frac{1}{2})\\ C_j^{-1}w_j(w_{j0}+\frac{1}{2})^{-1}\end{bmatrix}
 \end{displaymath}
 from which we have $w_j=C_j(w_{j0}+\tfrac12)d_j$ , $w_{j0}=\tfrac12(1+d_{j0})(1+d_{j0})^{-1}$ and
\begin{equation} \label{S}
S_j:=C_j^{-1}(C_j-I)=\begin{bmatrix}d_{j0}&~&~\\\vdots&\ddots&~\\d_{jn_{j}-2}&\cdots&d_{j0}\end{bmatrix}.        
\end{equation}
 Consequently, 
\begin{equation}
\label{wd}
\begin{split}
w_j&=(I-S_j)^{-1}(1-d_{j0})^{-1}d_j\\
d_j&=(w_{j0}+\tfrac12)^{-1}C_j^{-1}w_j
\end{split}
\end{equation}
Moreover
\begin{subequations}\label{DjD}
\begin{equation}\label{Dj}
D_{j}=\begin{bmatrix}
d_{j0}&~&~&~\\
d_{j1}&d_{j0}&~&~\\
\vdots&\ddots&\ddots&~\\
d_{jn_{j}-1}&\cdots&d_{j1}&d_{j0}
\end{bmatrix}
\end{equation}
and
\begin{equation}
\label{D}
D=\text{diag}\,(D_0,\dots,D_m) =(W+\tfrac{1}{2}I)^{-1}(W-\tfrac{1}{2}I).
\end{equation}
\end{subequations}
Therefore \eqref{uU} yields
\begin{equation}
\label{u}
u=Md,
\end{equation}
where $d$ is the $n$-vector satisfying
\begin{displaymath}
d=\begin{bmatrix}d_0'&d_{10}'&d_1'&\cdots&d_{m0}'&d_m'\end{bmatrix}',
\end{displaymath}
 and $M$ is the $n\times n$ matrix obtained by deleting the first row and the first column in $V^{-1}$. Since $w$ and $u$ have the same dimension $n$, the smooth maps \eqref{wd} together with \eqref{u} defines a diffeomorphic map from $w$ to $u$. Moreover, in view of \eqref{DjD}, there is a linear map $N$ such that $D=N(d)=N(M^{-1}u)$, and hence there is a linear map $L$ such that  $U=Lu$.
\end{proof}

\section{Main theorems}\label{sec: theorems}

Let $\mathcal{S}_{n}$ be the space of Schur polynomial of the form \eqref{sigma}, and let $\mathcal{P}_n$ be the $2n$-dimensional  space of pairs $(a,b)\in\mathcal{S}_{n}\times\mathcal{S}_{n}$ such that $f=b/a$ is positive real. Moreover, for each $\sigma\in\mathcal{S}_{n}$, let 
$\mathcal{P}_n(\sigma)$ be the submanifold of $\mathcal{P}_n$ for which \eqref{ab2sigma} holds. (Note that $\rho^2$ is the appropriate normalizing scalar factor once $(a,b)$ has been chosen.) It was shown in \cite{b8} that  $\{\mathcal{P}_n(\sigma)\mid \sigma\in\mathcal{S}_{n}\}$ is a {\em foliation\/} of $\mathcal{P}_n$, i.e., a family of smooth nonintersecting submanifolds, called {\em leaves}, which together cover  $\mathcal{P}_n$. Finally, let $\mathcal{W}_+$ be the space of all $w$ such that $\Sigma$ in \eqref{Pick} is positive definite.

\begin{theorem}\label{difthm1}
Let $\sigma\in\mathcal{S}_n$. Then for each $w\in\mathcal{W}_+$ there is a unique $(a,b)\in\mathcal{P}_n(\sigma)$ such that \eqref{varphi} satisfies the interpolation conditions \eqref{interpolation2} and the positivity condition \eqref{ab2sigma}. In fact, the map sending $(a,b)\in\mathcal{P}_n(\sigma)$ to $w\in\mathcal{W}_+$ is a diffeomorphism.
\end{theorem}

\begin{proof}
The Carath{\'e}odory function \eqref{varphi} can be written
\begin{displaymath}
\varphi(z)=\int_{-\pi}^\pi \frac{e^{i\theta}+z}{e^{i\theta}-z}\,\text{Re}\{\varphi(e^{i\theta})\}\frac{d\theta}{2\pi},
\end{displaymath}
where $(e^{i\theta}+z)(e^{i\theta}-z)^{-1}$ is a Herglotz kernel.
Moreover, differentiating we obtain
\begin{displaymath}
\varphi^{(k)}  (z)=\int_{-\pi}^\pi \frac{2e^{i\theta}}{(e^{i\theta}-z)^{k+1}}\,\text{Re}\{\varphi(e^{i\theta})\}\frac{d\theta}{2\pi}. 
\end{displaymath}
Therfore the interpolation problem can be formulated as the generalized moment problem to find the Carath{\'e}odory function \eqref{varphi} satisfying the moment conditions
\begin{equation}
\label{momentcond}
\int_{-\pi}^\pi \alpha_{jk}(e^{i\theta})\,\text{Re}\{\varphi(e^{i\theta})\}\frac{d\theta}{2\pi}=w_{jk},
\end{equation}
where
\begin{align*}
    \alpha_{j0}(z)&=\frac{z+z_j}{z-z_j}\quad j=0,1,\dots,m   \\
    \alpha_{jk}(z)&=\frac{2z}{(z-z_j)^{k+1}}\quad j=0,\dots,m, \, k=1,\dots, n_{j-1}  
\end{align*}
(see, e.g., \cite{BLkimura}). Then the statement of the theorem follows from \cite[Theorem 3.4]{KLR}.
\end{proof}

Next let $\Pi$ be the space of $n\times n$ symmetric, positive semi-definite matrices $P$ such that $h'Ph<1$. Moreover, for any fixed $\sigma\in\mathcal{S}_n$, define the rational map 
\begin{displaymath}
\Psi(w,P):=P-\Gamma(P-PhhP)\Gamma'-g(P)g(P)'
\end{displaymath}
on $\mathcal{W}_+\times\Pi$. Then the zero locus
\begin{displaymath}
\mathcal{Z}:=\Psi^{-1}(0)\subset \mathcal{W}_+\times\Pi
\end{displaymath}
is the solution set of \eqref{newCCE}. Following \cite{BFL} we define the projection
\begin{displaymath}
\pi_{\mathcal{W}_+}(w,P)=w
\end{displaymath}
restricted to $\mathcal{Z}$. Then there exists a solution to CEE if and only if $\pi_{\mathcal{W}_+}$ is surjective, and this solution is unique if and only if $\pi_{\mathcal{W}_+}$ is injective. Then we have the following counterpart of Theorem 1 in \cite{BFL}.

\begin{theorem}\label{difthm2}
The zero locus $\mathcal{Z}$ of the CEE \eqref{newCCE} is a smooth semialgebraic manifold of dimension $n$.  Morerover, $\pi_{\mathcal{W}_+}$ is a diffeomorphism between $\mathcal{Z}$ and $\mathcal{W}_+$. In particular, the CEE \eqref{newCCE} has a unique solution $P$ for each $(\sigma,w)\in\mathcal{S}_n\times\mathcal{W}_+$. Finally, the unique solution of the interpolation problem of Theorem~\ref{difthm1} is given by
\eqref{P2ab}, and 
\begin{equation}
\label{rankP=degf}
\text{rank\,} P =\text{deg\,} \varphi =\text{deg\,}f .
\end{equation}
\end{theorem}

\begin{proof}
First note that any solution $P$ of \eqref{newCCE}  is completely determined by the $n$-vector $p:=Ph$, so the dimension of the space $\Pi$ is $n$. It was shown in \cite{BLpartial} that \eqref{newCCE} can be reformulated as 
\begin{equation}
\label{ab2P}
P-JPJ'= -\frac12(ab' +ba')+\rho^2\sigma\sigma,
\end{equation}
where $a$ and $b$ are given by \eqref{Psigma2ab}. Note that this is independent of the fact that our new problem has different $(u,U)$. Since $J$ is a stability matrix, there is a unique solution $P$ for each $(a,b)\in\mathcal{P}_n(\sigma)$. The normalization factor $\rho^2$ is a smooth function of $(a,b)\in\mathcal{P}_n(\sigma)$ via \eqref{ab2sigma}. Therefore, by Theorem~\ref{difthm1}, the right member of \eqref{ab2P} is a smooth function of $w\in\mathcal{W}_+$, and, by elementary theory for the Lyapunov equation, so is  $P$. Consequently, $\pi_{\mathcal{W}_+}^{-1}$ is smooth, and since $\pi_{\mathcal{W}_+}$ is also smooth, it is a diffeomorphism. Moreover, since $\mathcal{Z}$ is the graph in  $\mathcal{W}_+\times\Pi$ of a smooth map defined on $\mathcal{W}_+$, it is a smooth manifold of the same dimension as $\mathcal{W}_+$, namely $n$. Finally, \eqref{rankP=degf} was established in \cite{BLpartial}. 
\end{proof}

\section{Solving CEE by homotopy continuation}\label{sec:homotopy}

The problem at hand is to solve the Covariance Extension Equation (CEE) \eqref{newCCE}
for the case that $u=\omega(w)$ is  a diffeomorphic function of the data $w$ and $U=Lu$, where $L$ is a linear map (Proposition~\ref{w2uUprop}).
If $u=0$, CEE takes the form
\begin{equation}
\label{CCEu=0}
P = \Gamma (P-Phh'P) \Gamma',
\end{equation}
which has the unique solution $P=0$. We would like to make a continuous deformation of $u$ to go between the solutions of \eqref{newCCE} and \eqref{CCEu=0}. To this end, we choose
\begin{equation}
\label{deformation}
u(\lambda)=\lambda u, \quad \lambda\in [0,1]. 
\end{equation}

\begin{proposition}
Let $\omega$ be the diffeomorphism in Proposition~\ref{w2uUprop}. Then $w(\lambda):=\omega^{-1}(\lambda u)\in\mathcal{W}_+\,$ for all  $\lambda\in [0,1]$.
\end{proposition}

\begin{proof}
It follows from  \eqref{D} that $W=(I-D)^{-1}-\tfrac12 I$, and therefore the corresponding deformation is 
\begin{displaymath}
W(\lambda)=(I-\lambda D)^{-1}-\frac12 I . 
\end{displaymath}
We want to show that $W(\lambda)$ satisfies $\Sigma >0$ in \eqref{Pick} for all  $\lambda\in [0,1]$. To this end, we note that a straightforward calculation yields
\begin{align*}
\label{lambdaPick}
    \Sigma(\lambda)&:=W(\lambda)E+EW(\lambda)^*    \notag \\
    & =(I-\lambda D)^{-1}(E-\lambda^2 DED^*)(I-\lambda D^*)^{-1} .
\end{align*}
However, $E-\lambda^2 DED^* \geq E- DED^* >0$ for all  $\lambda\in [0,1]$, and consequently $\Sigma(\lambda)>0$ as claimed.
\end{proof}

Consequently the equation 
\begin{equation*} \label{CEElambda}
 \hat{H}(P,\lambda) := P - \Gamma (P-Phh'P) \Gamma' - g(P,\lambda)g(P,\lambda)' =0
\end{equation*}
 with
\begin{equation*}\label{g(P)} 
g(P,\lambda)= u(\lambda) +U(\lambda)\sigma + U(\lambda)\Gamma Ph
\end{equation*}
has a unique symmetric, positive semidefinite solution $P(\lambda)$ with the property $h'P(\lambda)h<1$. 
The function $\hat{H}$ sending $(P,\lambda)$ to $\mathbb{R}^{n\times n}$ is a homotopy between \eqref{newCCE} and \eqref{CCEu=0}.  By Theorem~\ref{difthm2}, the trajectory $\{P(\lambda)\mid \lambda\in [0,1]\}$ is continuously differentiable and has no turning points and bifurcations \cite{Alexander}. This allows us to use  homotopy continuation to construct a computational procedure.

However, once $p:=Ph$ is known, CEE reduces to a Lyapunov equation of the type
$P=\Gamma P\Gamma' + Q(p)$,
which has a unique solution since $\Gamma$ is a stability matrix. Therefore \eqref{PgCCE} can be reduced from an algebraic equation with $\tfrac12 n(n+1)$ variables to one with $n$. In fact,
multiplying \eqref{ab2P} by $z^{j-i}$ and summing over all $i,j=1,2,\dots,n$ we recover \eqref{ab2sigma}, which in matrix form can be written
\begin{equation}
\label{redequ}
S(a)\begin{bmatrix}1\\b\end{bmatrix}=2(1-h'p)\begin{bmatrix}s\\\sigma_n\end{bmatrix} ,
\end{equation}
where
\begin{equation*}
S(a)=
\begin{bmatrix}
1&\cdots&a_{n-1}&a_{n}\\
a_{1}&\cdots&a_{n}\\
\vdots&\iddots\\
a_{n}
\end{bmatrix}+\begin{bmatrix}
1&a_{1}&\cdots&a_{n}\\
~&1&\cdots&a_{n-1}\\
~&~&\ddots&\vdots\\
~&~&~&1
\end{bmatrix}
\end{equation*}
and
\begin{equation*}
s=\begin{bmatrix}
1+\sigma_{1}^{2}+\sigma_{2}^{2}+\cdots+\sigma_{n}^{2}\\
\sigma_{1}+\sigma_{1}\sigma_{2}+\cdots+\sigma_{n-1}\sigma_{n}\\
\vdots\\
\sigma_{n-1}+\sigma_{1}\sigma_{n}
\end{bmatrix}.
\end{equation*}
However the last of the $n+1$ equations \eqref{redequ} is redundant \cite{BFL} and can be removed. Then we are left with $n$ equations
\begin{equation}\label{reducedCCE}
\begin{bmatrix}I_n&0\end{bmatrix}S(a)\begin{bmatrix}1\\b\end{bmatrix}=2(1-h'p)s
\end{equation}
in $n$ variables $p_1,p_2,\dots, p_n$.

Therefore we shall instead use the homotopy 
\begin{equation}
\label{ }
\begin{split}
H(p,\lambda):= &\begin{bmatrix}I_n&0\end{bmatrix}S(a(p,\lambda))\begin{bmatrix}1\\b(p,\lambda)\end{bmatrix}\\
&-2(1-h'p)s =0
\end{split}
\end{equation}
where
\begin{subequations}\label{p2ab}
\begin{equation}
a(p,\lambda)=(I-\lambda U)(\Gamma p+\sigma)-\lambda u
\end{equation}
\begin{equation}
b(p,\lambda) =(I+\lambda U)(\Gamma p+\sigma)+\lambda u ,
\end{equation}
\end{subequations}
which also has a unique solution  $p(\lambda)$ for all $\lambda\in [0,1]$.

From the implicit function theorem we obtain the differential equation
\begin{equation}
\label{diffequ}
\frac{dp}{d\lambda}=\left[\frac{\partial H(p,\lambda)}{\partial p}\right]^{-1}\frac{\partial H(p,\lambda)}{\partial\lambda}, \quad p(0)=0,
\end{equation}
where
\begin{align*}
 \frac{\partial H(p,\lambda)}{\partial\lambda}  
   & =\begin{bmatrix}I_n&0\end{bmatrix}(S(a(p,\lambda))-S(b(p,\lambda)))\begin{bmatrix}0\\g(p,1)\end{bmatrix} \\
\frac{\partial H(p,\lambda)}{\partial p}  
&=\begin{bmatrix}I_n&0\end{bmatrix}(S(a(p,\lambda))+S(b(p,\lambda)))\begin{bmatrix}0\\\Gamma\end{bmatrix}+2hs'\\
&+\begin{bmatrix}I_n&0\end{bmatrix}(S(a(p,\lambda))-S(b(p,\lambda)))\begin{bmatrix}0\\ \lambda U\Gamma\end{bmatrix}
\end{align*}
and
\begin{equation}\label{g(p)} 
g(p,\lambda)= u(\lambda) +U(\lambda)\sigma + U(\lambda)\Gamma p. 
\end{equation}
The differential equation \eqref{diffequ} has a unique solution $p(\lambda)$ on the interval $\lambda\in[0,1]$, and the unique solution of the Lyapunov equation
\begin{equation}
\begin{split}
&P-\Gamma P\Gamma'=-\Gamma p(1)p(1)'\Gamma'+\\
&\qquad (u+U\sigma+U\Gamma p(1))(u+U\sigma+U\Gamma p(1))'
\end{split}
\end{equation}
is the unique solution of \eqref{newCCE}.
To solve the differential equation \eqref{diffequ} we use  predictor-corrector steps \cite{AllgowerGeorg}. We leave the details of this to another paper.

\subsection*{A numerical example}

To illustrate our numerical procedure and demonstrate its robustness and efficiency we consider  a problem where the system have poles close to the unit circle, a situation for which methods based on convex optimization has had problems. 
{\small
Given the eight pairs of interpolation data \cite{Nagamune robust solver}
\begin{equation*}
\begin{split}
\{z_{0},\cdots,z_{n}\}=&\{\infty,0.8709-0.8967i,0.8709+0.8967i,\\
&0.3344-1.2044i,0.3344+1.2044i,1.1,\\
&-0.6474+0.8893i,-0.6474-0.8893i\}\\
\{w_{0},\cdots,w_{n}\}=&\{0.5,0.7973+0.2568i,0.7973-0.2568i,\\
&0.5451 + 0.3645i,0.5451 - 0.3645i,0.7693,\\
&0.7693 + 0.7693i,0.7693 - 0.7693i\} ,
\end{split}
\end{equation*}
}
for which \eqref{Pick} is positive definite, and the spectral zeros
$\{0.95e^{\pm1.22i},0.95e^{\pm2.3i},\pm0.99i,-0.99\}$, defining $\sigma$, 
we obtain the unique solution $f(z)=b(z)/2a(z)$ of degree 7 with
\begin{equation*}
\begin{split}
b(z)=&z^{7}-1.364z^{6}+1.112z^{5}-0.3812z^{4}\\
&-0.4479z^{3}+1.119z^{2}-1.412z+0.8781\\
a(z)=&z^{7}-1.771z^{6}+1.815z^{5}-1.205z^{4}\\
&1.28z^{3}-1.814z^{2}+1.773z-0.8775\\
\end{split}
\end{equation*}

Fig.~\ref{figure1} shows how the trajectories of the poles, i.e., the zeros of $a(p(\lambda))$, move as $\lambda$ varies from $0$ to $1$. The poles for $\lambda =0$ are marked with circles and the ones for $\lambda=1$ by $\times$. 
\begin{figure}[thb!]
  \centering
  \includegraphics[width = 0.5\linewidth]{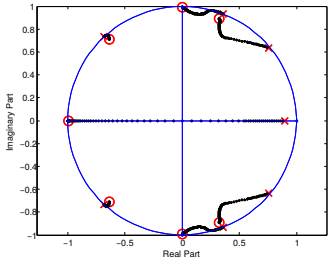}
  \caption{The trajectories of the poles}
  \label{figure1}
\end{figure}
Several of the zeros of $a(p(1))$, i.e., the poles of the final solution, are seen to be situated 
very close to the unit circle. This is a situation that is hard to solve numerically by the usual convex optimization methods. 

\section{Some applications to systems and control}\label{sec:applications}


\subsection{Spectral estimation with model reduction}\label{modred}
Generate an observed  time series $y_{0},y_{1},y_{2},\cdots,y_{N}$ by passing normalized white noise through a filter with the transfer function $w(z)=\sigma(z)/a(z)$ and then in turn passing this time series through the bank of filters 
\begin{figure}[!htp]
  \centering
  \includegraphics[width = 0.7\linewidth]{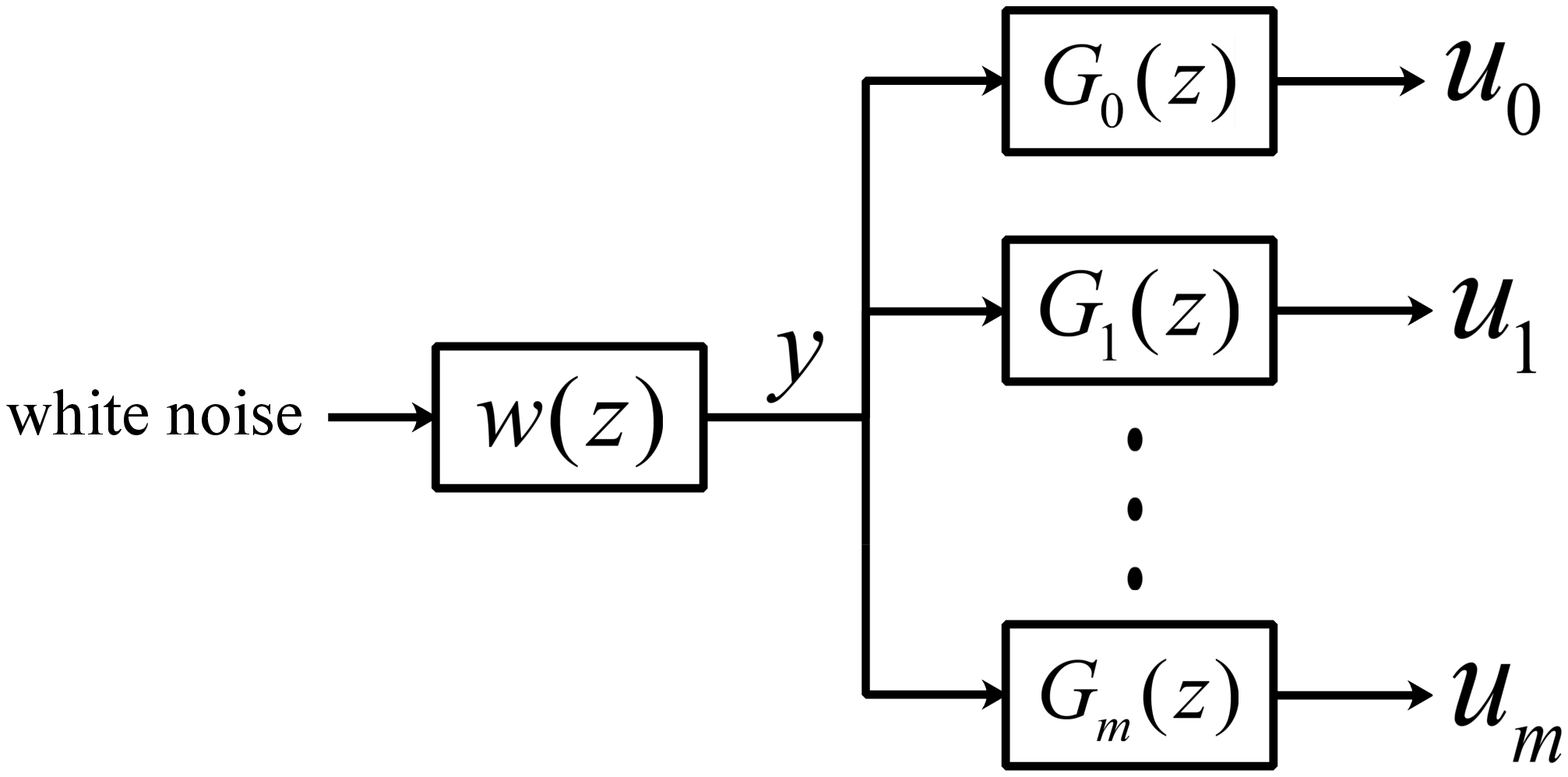}
  \label{procedure}
\end{figure}
\begin{equation}
G_{j}(z)=z(zI-Z_{j})^{-1}e^{n_{j}},\quad j=0,1,\cdots,m .
\end{equation}

where $Z_j$ is given by \eqref{Z}. 
The covariance matrix $\Sigma:=\mathbb{E}\{u(t)u^{*}(t)\}$ can be estimated from the observed output of the bank of filters, and then the matrix $W$ in \eqref{W} can be estimated from the Lyapunov equation
\begin{displaymath}
WE+EW^{*} =\Sigma.
\end{displaymath}
Cf. \eqref{Pick}, where $\Sigma$ is a state covariance \cite{b16}. After estimating  $W$ from data, we then apply our algorithm to solve the corresponding problem \eqref{interpolation}. 
We choose a transfer function $w(z)$ of degree six with the zeros at $0.92e^{\pm1.5i},0.49e^{\pm1.4i},0.95e^{\pm2.5i}$ and poles at $0.8e^{\pm2.1i},0.83e^{\pm1.34i},0.76e^{\pm0.8i}$.  Determining $W$ from the bank of filters, our method produces the  power spectral density shown in Fig.~\ref{figure5}, which is almost identical to the true one (also depicted). From the left picture in Fig.~\ref{figure6} we see that there is no close zero-pole cancellation.
However, the singular values of $P$ are
\begin{equation*}
2.0170,\, 0.4184,\, 0.02585,\, 0.01858,\, 0.005741,\, 0.002466 ,
\end{equation*}
where the last two are vey small, so the positive degree is close to four. Therefore, using the dominant spectral zeros at $0.92e^{\pm1.5i},0.95e^{\pm2.5i}$ only, the singular values becomes
\begin{equation*}
1.2205,\, 0.2913,\, 0.01605,\, 0.02563 .
\end{equation*} 
The estimated spectral density of the reduced order system of degree four is depicted in Fig.~\ref{figure5} and shows little difference from the one of degree six. 
\begin{figure}[thb!]
  \centering
  \includegraphics[width = 0.7\linewidth]{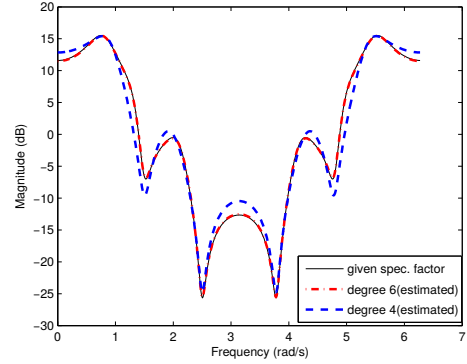}
  \caption{Estimated spectral densities and the true one}
  \label{figure5}
\end{figure}
However, the location of zeros and poles for the reduced-order system, shown to the right in Fig.~\ref{figure6}, are quite different from those of the 6-order system. 

\subsection{Robust control}\label{robustex}
Consider the feedback configuration
\begin{figure}[thb!]
\centering
\includegraphics[width=0.7\linewidth]{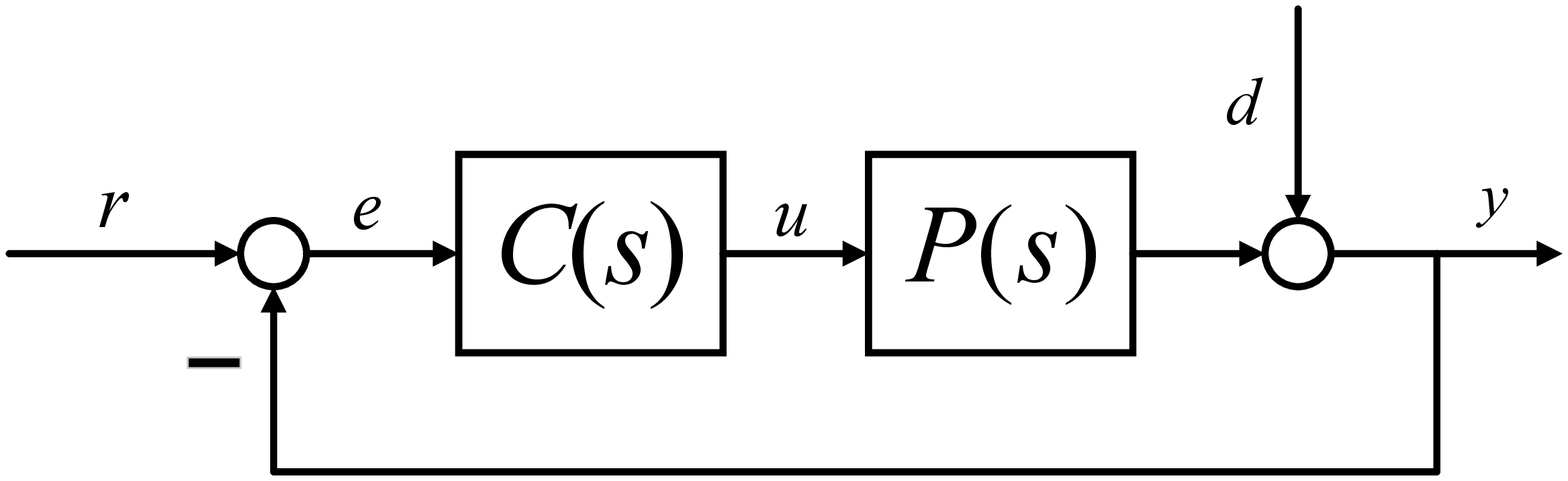}
\label{robust}
\end{figure}
\begin{figure}[thb!]
  \centering
  \includegraphics[width = 0.4\linewidth]{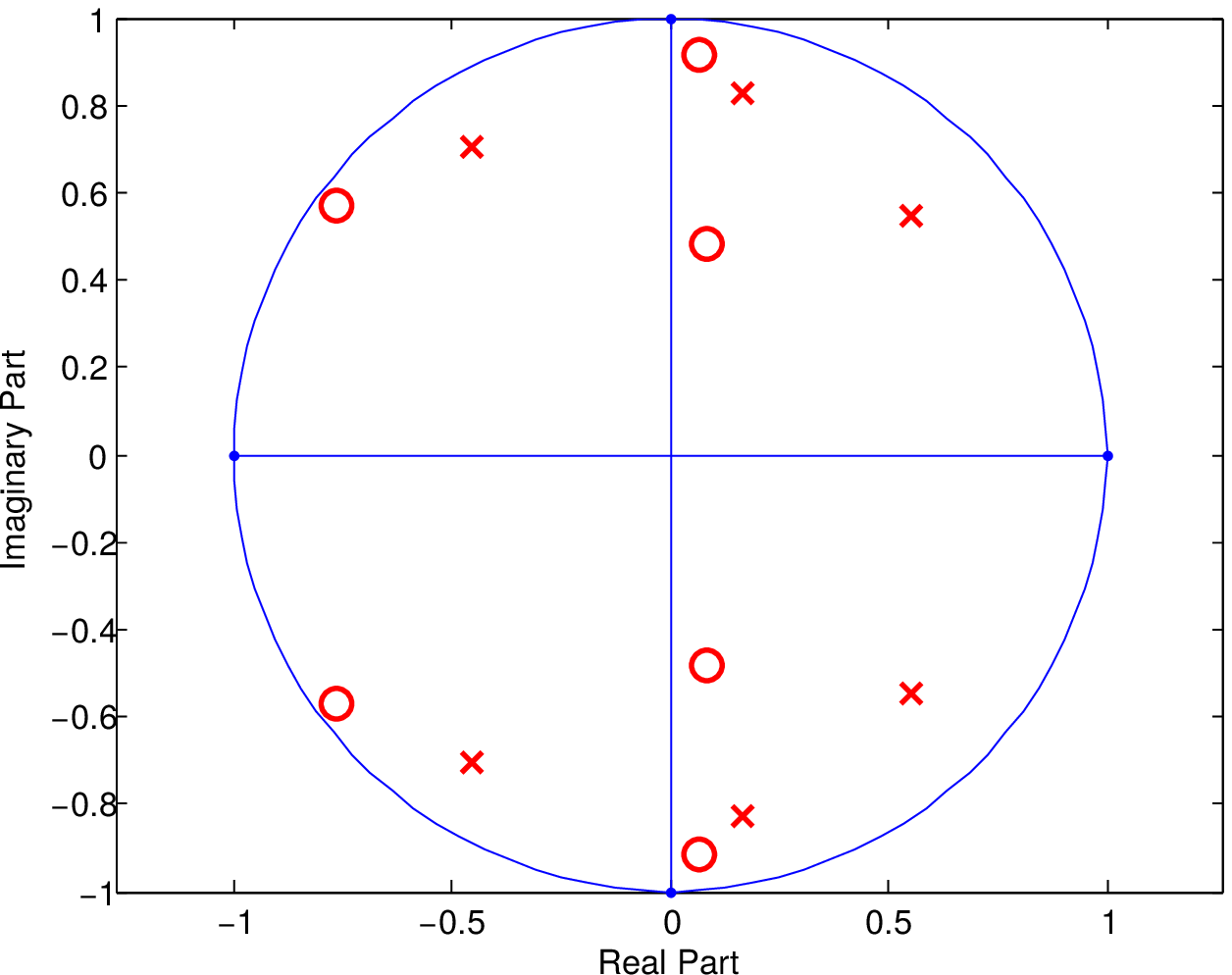}
  \label{figure3}
\quad
  \includegraphics[width = 0.4\linewidth]{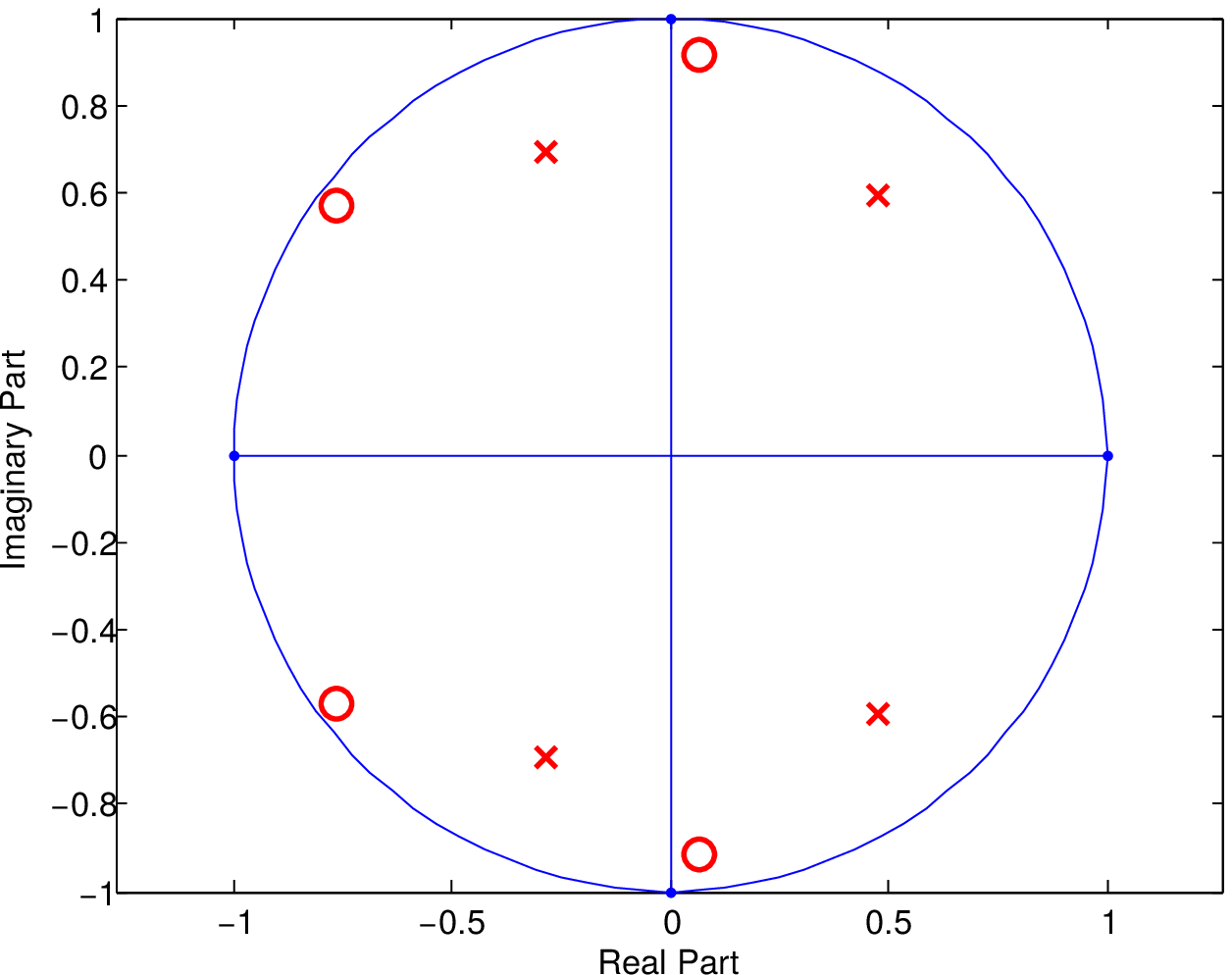}
  \caption{Zeros $ (\circ) $ and poles $ (\times) $ of original estimated system (left) and the reduced-order one (right).}
  \label{figure6}
\end{figure}

where $ r$ is the reference input and $d$ is the disturbance on the output $y$. Given an unstable plant 
\begin{displaymath}
P(s)=\frac{-8s^{2}+62s+200}{10s^{4}+8s^{3}+7s^{2}+0.5s},
\end{displaymath}

design a strictly proper controller $C(z)$ such that the feedback system satisfies the design specifications: (i) The system is internally stable. For a step reference $r$, (ii) the settling time is about 8 seconds, (iii)  the overshoot is at most 10$\%$, and (iv) the control signal $u(t)$ has magnitude  at most 0.5.
This design problem is similar to the one considered in \cite{DFT} using the classical central solution and in \cite{blomqvist} using a homotopy method to solve the convex optimization problem. Here we show how to shape the frequency response of the sensitivity function $S(s):=(1+P(s)C(s))^{-1}$ by just changing the spectral zeros.

The plant has one unstable pole at $s=0$, and two unstable zeros at $s=\infty$ and $s=10.2008$ with multiplicity two and one respectively.  Therefore the sensitivity function must satisfy the interpolation conditions
\begin{displaymath}
S(0)=0,\,  S(\infty)=1,\,S'(\infty)=0, \, S(10.2008)=1.
\end{displaymath}
Moreover, to ensure that $C$ is strictly proper we must have
\begin{displaymath}
S''(\infty)=0.
\end{displaymath}
See, e.g., \cite{DFT}. From the design specifications (ii) and (iii) we  can obtain an approximately ideal sensitivity function
\begin{equation}\label{Sideal}
S_{idel}(s)=\frac{s(s+0.9)}{s^{2}+0.9s+0.75^2}
\end{equation}
of second order. However, \eqref{Sideal} cannot be used since it does not satisfy all the interpolation conditions. For disturbance attenuation we also need a condition 
\begin{equation*}
\lVert S \rVert_{\infty}<\gamma .
\end{equation*}

Using the M{\"o}bius transformation  $z=\tfrac{10}{9}(1+s)(1-s)^{-1}$,
which maps the points in the right half plane into the exterior of the unit disc, the problem is reduced to finding a  function $f(z)=(\gamma+S(z))(\gamma-S(z))^{-1}$ that is positive real and
satisfies
\begin{equation*}
\begin{split}
&f(\tfrac{10}{9})=1,\quad f(-1.3526)=\frac{\gamma+1}{\gamma-1}\\
&f(-\tfrac{10}{9})=\frac{\gamma+1}{\gamma-1},\quad f'(-\tfrac{10}{9})=0,\quad f''(-\tfrac{10}{9})=0
\end{split}
\end{equation*}
Since there are five interpolation conditions, we can construct an interpolant of degree four by choosing four spectral zeros. We choose $\gamma=1.8$ and spectral zeros at $\pm 0.9i,5,\infty$. (More details on how to choose these parameters can be found in \cite{Nagamune R}).
Our computational procedure lead to the controller
\begin{equation*}
C(s)=\frac{6.986 s^3 + 5.589 s^2 + 4.89 s + 0.3493}{s^4 + 21.43 s^3 + 144.9 s^2 + 336.2 s + 233}.
\end{equation*}
The settling time is $ 6.55\, s$, the overshoot is 8.86$\%$, and the largest magnitude of $u$ is $0.13$, which all satisfy the design specifications. 
Fig.~\ref{robust1} shows the frequency response of $S_{ideal}$ and $S_{computed}$, which show little  difference. 
\begin{figure}[thb!]
\centering
\includegraphics[width=0.7\linewidth]{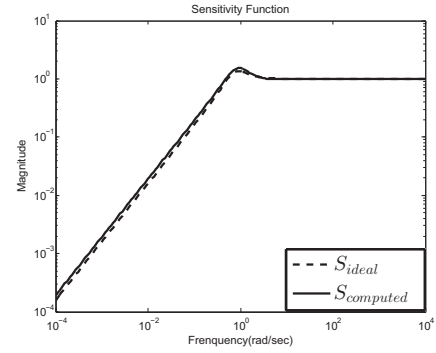}
\caption{Frequency responses of $S_{ideal}$ and $S_{computed}$}
\label{robust1}
\end{figure}

\subsection{Comparing our approach to other methods}
Our new method shares the advantage of the convex optimization methods initiated in \cite{BGuL,b1} in that the solutions can be smoothly tuned by choice of spectral zeros. The classical method \cite{DFT} produces a controller of degree 8 for our robust control example in Section~\ref{robustex}, whereas ours is degree 4 and the design specifications are satisfied with larger margins. However, solving the convex optimization problem by Newton's method when the system has poles close to the unit circle (a common situation) is problematic.  This disadvantage was overcome in \cite{Enqvist,Nagamune robust solver} by solving the optimization problem using homotopy from an initial solution. Our method has the additional advantage that there is no need to determine an initial solution. Moreover, as illustrated in Section~\ref{modred}, in our method one can directly detect the possibility of model reduction by simply checking the (approximate) rank of the solution $P$ of the Riccati-type equation \eqref{newCCE}.

\balance

\end{document}